\documentclass[10pt,a4paper]{amsart}
\usepackage[utf8]{inputenc}
\usepackage[english]{babel}
\usepackage{amsmath}
\usepackage{amsfonts}
\usepackage{amssymb}
\usepackage{amsthm}
\usepackage[foot]{amsaddr}

\author{Laurent Fallot}
\address{EPOC - UMR CNRS 5805 - Equipe PROMESS, Université de Bordeaux - Bordeaux-INP, Avenue des Facultés, CS60099, 33400 Talence, France}
\email{Laurent.Fallot@bordeaux-inp.fr}

\title{Notes on proof by dichotomy}
\date{Revision 2.2 \today}

\newtheorem{definition}{Definition}
\newtheorem{theorem}{Theorem}

\newtheorem{coro}{Corollary}[theorem]

\def\setN{\mathbb{N}}

\begin{document}
\begin{abstract}
In this document we define a method of proof that we call proof by dichotomy. 
Its field of application is any proposition on the set of natural numbers $\setN$. 
It consists in the repetition of a step. 
A step proves the proposition for half of the members of an infinite subset $U$ of $\setN$ members for which we neither know if the proposition is verified nor not. 
We particularly study the case where the elements of $U$ are separated by the parity of the quotient of euclidean division by $2^k$. 
In such a case, we prove that if a natural $n$ does not verify the proposition, then it is unique.
\end{abstract}
\maketitle
\section{Acknowledgement}
While working on Collatz problem, we introduced this method of proof.
We decide to write some personal notes on it. 
Having dropped this field of interest for a long time we have not an up to date bibliography on this field of mathematics. 
In the doubt about its knowledge, we choose to share these informal notes with the research community. 
We apologize if this work is well known. 
We do not intend to steal other one works.

\section{Introduction}
It is common to use a proof by induction in many sciences.
It is a very well known method of proof so we only recall its principle. 

This method of proof is used to validate a given proposition $P$ on the set of positive integers $\setN$.
It starts with the proof of a base case. Often it's 0 but not only. This done, we have to build the induction step which consists in being able to prove $P(N)$ independently of the value of $N$ from the hypothesis that for all $n<N,~P(n)$ is valid.

Until today, every attempt to use this method of proof to validate Collatz problem for instance failed. 
Many reason of these failures may be given. 
Mainly, to simplify, it seems that for any $N$ we may find a new length record greater than $N$.

Today, many publications about that problem tend to better the density of solved cases in $\setN$. 
Suppose that we get by this way a density of 1. 
Can we conclude that Collatz problem is solved ? 
It is the question we are trying to answer.

In this document, we introduce a new method of proof that we call proof by dichotomy.
We start with a subset of $\setN$ denoted by $U_0(P)$. 
This set is composed of all the natural numbers for which we want to prove proposition $P$. 
As in a proof by induction, a proof by dichotomy is an infinite sequence of a step. 
But it differs from a proof by induction in that during the $k^{th}$ step we split $U_{k-1}(P)$ into two equally distributed parts and we prove that $P$ is verified on exactly one of these parts. For the other one, $P$ stays unsolved.

So, we maintain along the proof two sets : $U_k(P)$ that contains all the numbers for which we don't know yet if they verify $P$ or not and $S_k(P)$ the set of numbers $n$ for which we proved $P(n)$. 

Note that during the step $k$, we can use the property that separate $U_{k-1}(P)$ into two parts and the fact that $P(n)$ is true for any $n\in S_{k-1}(P)$ to validate $P$ on one of the two parts.

A quick statistical study tends to say that if we have a proof of a proposition $P$ on the set $\setN$ by this way, then $P$ is verified for every natural numbers. 
But a closer look at this shows it is false in the general case. 

We particularly study an approach driven by the normalisation of any number as $2^k q + r$ where $r<2^k$. We prove that there can exist a natural number that does not verify the property $P$ we are trying to prove but this number is unique in this case.

Further, we give indications on when this counter-example of $P$ exists and what it can be.

\section{Position of the problem}
The term proposition denotes any application defined from $\setN$ into the boolean algebra $\{True, False\}$.

Let $n$ be a natural number and $P$ a proposition. We say that
\begin{itemize}
\item $n$ verifies or satisfies $P$ when $P(n)=True$,
\item $n$ denies or refutes $P$ when $P(n)=False$ and
\item $n$ is an unsolved case for $P$ when we don't know the value of $P(n)$ at a given step of a proof.
\end{itemize}

A counter-example of a proposition $P$ is a natural number that refutes $P$.

We define now what we call a "proof by dichotomy" on the natural numbers set $\setN$.

\begin{definition}[Proof by dichotomy]
Let $P$ be a proposition we want to prove on the natural numbers set $\setN$ or on one of its infinite subsets. 
The proof by dichotomy method is an infinite application of a step. This step consists into splitting the set of unsolved cases into two equally distributed parts and to be able to prove from the solved cases that the proposition is verified on exactly one of this part. Then, the second part is the set of unsolved staying cases given as the entry of next step.
\end{definition}

In fact, it is easy to see that this method of proof maintains two sequences of sets : 
\begin{description}
\item[a sequence \boldmath $\left ( U_k(p) \right ) _{k\in\setN}$] 
each of these $U_k(P)$ contains all the unsolved cases of $P$ after step $k$. It starts with $U_0(P)$ the infinite subset of $\setN$ on which we want to prove $P$.

\item[a sequence \boldmath $\left ( S_k(p)^{~} \right ) _{k\in\setN}$] each of these $S_k(P)$ is composed of solved cases during the $k$ first steps. We start with the empty set for $S_0(P)$.
\end{description}

At the $k^{th}$ step, we divide $U_{k-1}(P)$ in two equally distributed subsets~: $U_{k,1}(P)$ and $U_{k,2}(P)$. This partition of $U_{k-1}(P)$ is chosen so that we can prove $P$ is valid on only one of these subsets. The members of the second one stay unsolved cases.
This because, if we were able to either prove $P$ on both parts or prove $P$ on one of them and its counter-proposition on the other part, the proof is closed and out of interest here.

Let us say that we can prove that $P$ is valid on $U_{k,1}(P)$ without loss of generality.  
Then we know the result of $P(n)$ for every $n\in U_{k,1}(P)$ but we still don't when $n\in U_{k,2}(P)$. By this way, we can define $U_k(P) = U_{k,2}(P)$ and build $S_k(P) = S_{k-1}(P)\cup U_{k,1}(P)$.
More precisely, it is easy too see that any member of $S_k(P)$ verifies $P$ because we add to this set only the numbers for which we proved that they verify $P$. 
So, we can use this particularity on $S_{k-1}(P)$ to justify that $P$ is valid on $U_{k,1}(P)$.

This method of proof is not so far from a proof by induction. But, the induction step differs. Instead of proving that any $n<n_0$ verifies $P$ infers that $n_0$ verifies $P$ whatever $n_0$ is, we have to prove that from a set of unsolved natural numbers we can extract half of them by the use of the solved cases.

By this way, let us suppose that we start with $\setN$ as the first set of unsolved numbers against a given proposition $P$ to be proved. 
After a first step, we know that one out of two of them verifies $P$. But, we don't know the answer for the other half. 
A second step will lead to only a quarter of the numbers stay unsolved against $P$ and so on. 
So that, after the application of $k$ steps, we get that only ${1/{2^k}}^{th}$ of the natural numbers stay unsolved against $P$. 
Let us call this number the density of unsolved cases at step $k$. 
An infinite application of this step tends to prove that $P$ is verified for all the natural numbers. 
Indeed, the density of unsolved cases tends to $\lim\limits_{k \rightarrow \infty} 1/2^k = 0$. 
But what is it in fact ?

Consider that we have an initial infinite set of unsolved cases. 
Dividing it in two parts still gives an infinity of unsolved cases. 
This means that whatever the number of steps we apply, we always get an infinity of numbers that potentially deny $P$. 
Further, if we try to count the final number of cases that may refute $P$ we get $\lim\limits_{k \rightarrow \infty} \infty / 2^k = \infty/\infty$  which is undefined. So, the number of potential counter-examples may be anything : 0, any finite number or infinity.

Let us study a particular case of this method of proof driven by the binary code of a natural number.

\paragraph{Remark} Before going further, note that we focus this method of proof on infinite subsets because this is the most complex case. For finite sets, if we consider $N$ the initial number of unsolved cases, this finite number is divided by 2 at each step. The sequence obtained is a strictly decreasing sequence of positive numbers starting with a finite positive integer. So, in the worst case, we get at most one counter-example after something around $\log_2(N)$ steps where $\log_2$ identifies the logarithm in base 2 function.

\section{Proof by dichotomy driven by the euclidean division of natural numbers by $2^k$}
\subsection{Notations and conventions}
Following, we use the notation $a\,/\,b$ where 
$a,b \in \setN,\; b\neq 0$ 
to denote the quotient of the euclidean division of $a$ by $b$. 
On the same way, $a\,\%\,b$ identifies the remainder of this division also called remainder of $a$ modulo $b$.

By the definition of the euclidean division and for any $k \in \setN$, we can write any natural number $n$ as $n=2^k \cdot q_{k,n} + r_{k,n}$ where $q_{k,n} = n\,/\,2^k$ and $r_{k,n} = n\,\%\,2^k$. This normalization is unique. We use this property to build the dichotomy method we discuss further.

\subsection{Construction}
Let $P$ be a proposition we want to prove on $\setN$. Let $U_0(P)=\setN$ be the initial set of unsolved natural numbers. Note that we can also start with any infinite subset of $\setN$ if we need to prove $P$ only on this subset.
At step $k>0$, we divide $U_{k-1}(P)$ into two parts by the use of the parity of $n\,/\,2^k, n\in U_{k-1}(P)$. On one side, the $n$ having odd quotients. On the other side, the $n$ having even ones.
We suppose that we can prove $P$ on exactly one of these two subsets of $U_{k-1}(P)$.

This constructs a proof by dichotomy driven by the euclidean division of natural numbers by $2^k$.

Following, we prove that there cannot exist more than one counter-example of $P$.

\subsection{Validity of a proof by dichotomy driven by the euclidean division of natural numbers by \boldmath $2^k$}
We claim the following theorem

\begin{theorem}
\label{thm:general}
Let $P$ be a proposition on $\setN$. If it exists a proof of $P$ by dichotomy driven by the euclidean division of natural numbers by $2^k$, then $P$ is satisfied by any natural number but at most one.
\end{theorem}

Before we introduce the proof of theorem~\ref{thm:general}, remind that we maintain two sequences of sets along such a proof : $\left ( U_k(P)\right)_{k\in\setN}$, the set of unsolved cases and $\left ( S_k(P)\right)_{k\in\setN}$, the set of solved cases. Furthermore, for any $k$, the pair $(U_k(P),S_k(P))$ is a partition of the starting set $U_0(P)$.

Note now, that any member $n$ of $U_k(P)$ migrates to $S_{k+1}(P)$ if and only if we can prove that $n$ satisfies $P$. So, any member of $S_k(n)$ verifies $P$ whatever $k$ is.

Let us prove theorem~\ref{thm:general} now.

\begin{proof}
Under the conditions of theorem~\ref{thm:general}, let us suppose that there exist a finite natural number $n$ that refutes $P$ and prove that $n$ is the only one. 

As $n$ denies $P$, $n$ cannot be a member of any $S_k(P)$ because the members of these sets verify $P$ by the construction of this set. This infers that $n\in U_k(P)$ for any $k\in \setN$.

It is clear that for any natural number $m\neq n$, there exist at most one $k\in \setN$ such that $m = 2^k q_{k,m} + r_{k,m}$ and $n=2^k q_{k,n} + r_{k,n}$ with different parities of $q_{k,m}$ and $q_{k,n}$. Consider the smallest of these $k$. Thus we have $r_{k,n} = r_{k,m}$.

The natural number $n$ is supposed finite. Then we can define $k_0$ such that $2^{k_0}>n\geq 2^{k_0-1}$.
That is $n = 2^{k_0-1} + r$ where $r=n\,\%\,2^{k_0-1}$. Let us consider two cases : when $k\leq k_0$ and when $k > k_0$.

\paragraph{Case $k\leq k_0$}
We know that $n\in U_{k-1}(P)$ otherwise we should have proved that $n$ verifies $P$. On the same way, it is easy to see that $m\in U_{k-1}(P)$ as it shares the same remainder modulo $2^k$ with $n$. But $n\in U_{k}(P)$ inferring that we can prove $P$ for any member of $U_{k-1}(P)$ having a parity different from $q_{k,n}$ one by hypothesis. As $q_{k,m}$ and $q_{k,n}$ have different parities and we are not able to conclude that $n$ verifies or not $P$ at this step, we get that $m$ verifies $P$ by the definition of a proof by dichotomy.

\paragraph{Case $k > k_0$}
As in the previous case , $n\in U_{k-1}(P)$ and $m\in U_{k-1}(P)$ for the same reasons.
The defined $k_0$ verifies that $n = 2^{k_0-1} + r_{k_0-1,n}$. 
That infers that $q_{k,n} = 0$ thus that $q_{k,n}$ is necessarily even.
Furthermore, as $n$ refutes $P$, we are not able to prove that $P$ is satisfied for any even $q_k$, $k>k_0$. 
By deduction from the hypothesis, $P$ is satisfied by any natural number with an odd quotient by the euclidean division by $2^k$.
As $q_{k,n}$ and $q_{k,m}$ have opposite parities, $q_{k,m}$ is odd and $m$ satisfies $P$.

In conclusion, whatever the case is, if a natural number $m$ differs from $n$, it satisfies $P$ and $n$ is the only unsolved natural number against $P$.
So, if a number $n$ refutes $P$ under the conditions of theorem~\ref{thm:general} it's the only one. Finally, $P$ is verified for any natural number but at most one.
\end{proof}
Theorem~\ref{thm:general} claims that we cannot have more than one natural number that refutes $P$ once a proof of $P$ by dichotomy driven by the euclidean division of natural numbers by $2^k$ is established. 
In fact, the proof does not insure the existence of a counter-example. 
It just claims that there exist at most one unsolved natural number. Let us introduce a first corollary describing the consequences of the existence of a counter-example.
\begin{coro}
\label{cor:counter_example_exist}
Under the hypothesis of theorem~\ref{thm:general}, the existence of a counter-example infers that there exist $K\in \setN$ such that for any $k>K$ we are only able to prove $P$ when the quotient by euclidean division by $2^k$ is odd.
\end{coro}
This corollary may be easily deduced from the proof of theorem~\ref{thm:general}.
\begin{proof}
From the second case of the proof of theorem~\ref{thm:general} we get that if there exists an $n$ refuting $P$ and it is of the form  $n=2^{k_0}+(n\,\%\,2^{k_0})$ then for any $k>k_0,~n\,/\,2^k = 0$ so it is even. As $n$ cannot be in any $S_k(P)$, this infers that we are able to prove $P$ for any natural number $m$ verifying $m\,/\,2^k$ is odd and this whatever $k>k_0$ is. \end{proof}

From the proof of theorem~\ref{thm:general}, especially from the second case, it is easy to understand which number $n$ is unsolved. Effectively we can construct the binary code of $n$ by writing down from right to left 0 when we can prove $P$ for odd quotients and 1 otherwise while running the steps from the first one to the end.
The number $n$ is supposed finite so, let $k_0$ defined such that $n=2^{k_0} + (n\,\%\,k_0)$. Corollary~\ref{cor:counter_example_exist} claims that for every $k>k_0$, $P$ can be proved only for odd quotients. 
This corresponds to leading zeroes we can put on the binary code of $n$. 
So, we just have to follow the steps of the proof from the first one to the $k_0^{th}$ one.

Thus, we can conclude the proof of $P$ mainly in two ways.

On the one hand, we may be satisfied with a unique possible counter-example because it is out of interest for instance. In this case, we only need to know the value of $n$. The method to do this is described above.

On the other hand, if we need to finish the proof of $P$ on $U_0(P)$ we just have to deny the existence of a counter-example. For instance, we can proceed either by identifying this unsolved natural number and prove that it verifies $P$ in fact or by denying the existence of a unique case by any other way.

Following, we present some more corollaries of theorem~\ref{thm:general}. 
The first one may be perceived as a counter-proposition of corollary~\ref{cor:counter_example_exist}
\begin{coro}
\label{cor:no_counter_example}
Under the hypothesis of theorem~\ref{thm:general}, if for any $K\in\setN$ there exist at least a $k>K$ such that we can prove $P$ on even quotients by the division by $2^k$, then any finite natural number verifies $P$. That is, we can conclude that $P$ is verified over $U_0(P)$.
\end{coro}

\begin{proof}
The proof of this corollary is easy to get. 
Indeed, note that the hypothesis of corollary~\ref{cor:no_counter_example} denies the conclusion of corollary~\ref{cor:counter_example_exist}, refuting by the same way the hypothesis of this corollary which is the existence of a finite counter-example of $P$. So, we cannot have a counter-example under the hypothesis of corollary~\ref{cor:no_counter_example}
\end{proof}

The following corollary may be seen as the reverse proposition of corollary~\ref{cor:counter_example_exist} in the sense that it gives a sufficient condition to insure the existence of a finite unsolved case.

\begin{coro}
\label{cor:may_be_counter_example}
Under the hypothesis of theorem~\ref{thm:general}, if we can find a $K\in \setN$ such that for any $k>K$ we are only able to prove $P$ when the quotient by euclidean division by $2^k$ is odd, then there exist a unique unsolved case $n$. \end{coro}

\begin{proof}
By following the steps from the first to the $K^{th}$ we previously saw that we can build a finite unsolved $n$ for $P$. In fact, $U_K(P) = \{2^{K+1} a + n, a\in \setN\}$. The further steps retrieve from $U_K(P)$ progressively all of its elements having $a\neq 0$. In conclusion, after the infinite application of the step, only $n$ stays. This validates corollary~\ref{cor:may_be_counter_example}.
\end{proof}

Recall that the existence of an unsolved case doesn't imply the existence of a counter-example. So, under only the conditions of corollary~\ref{cor:may_be_counter_example} we cannot affirm the existence of a counter-example.

A particular case of corollary~\ref{cor:may_be_counter_example} is the following.

\begin{coro}
\label{cor:only_odd}
Under the hypothesis of theorem~\ref{thm:general}, suppose that at any step $k$ we can only prove $P$ when $n\,/\,2^k$ is odd. In such a case, the only possible counter-example for $P$ is 0. If 0 is not a member of the initial set then, $P$ is verified on $U_0(P)$.
\end{coro}

\begin{proof}
The proof of corollary~\ref{cor:only_odd} is immediate from corollary~\ref{cor:may_be_counter_example}. 
In fact it is corollary~\ref{cor:may_be_counter_example} with $K=0$. So, the natural number $n$ build in its proof is 0 inferring corollary~\ref{cor:only_odd} by the same way.
\end{proof}

\subsection{Conclusion of the validation}
We claimed that a proof by dichotomy driven by the euclidean division by $2^k$ can give at most one counter-example of $P$. On the one hand, if among the infinite application of the step, we can regularly prove that $P$ is satisfied for an even quotient, then this counter-example does not exist. On the other hand, if from a given $K$ we are only able to prove $P$ for numbers having an odd quotient, we may not have more than one unsolved case. The path followed during the $K$ first applications of the step  gives an indication on a potential counter-example. If there exists an unsolved case $n$, we are free to ignore $n$, to prove separately either $P(n)$ or not or to use another way to prove the non existence of a single counter example. 

Let us take Collatz problem as an example of usage of this method of proof.

\section{Example : application to Collatz problem}

We warn the reader that we do not intend to give neither a proof nor a refutation of Collatz problem. We just point out that if someone can build a proof of this problem by dichotomy driven by the euclidean division by $2^k$, then it will be solved for any natural number different from 0.

\subsection{Definition of the problem}
The Collatz problem is known under many different names. Some call it "the 3x+1 problem" , "Syraccuse problem" or Ulam problem and so on. 

This problem is easy to define and looks like an easy one but many scientists proposed some other formulations of it expecting that a proof will come from this model but none leads to a positive result. Neither a negative one. So it is still an open problem. 
Firstly, let us recall it.

Collatz defined a function on natural numbers that we shall denote by $C(n)$. We can define it by $C(n) = n/2$ if $n$ is even and $C(n) = 3n+1$ if $n$ is odd. There exist some different formulations of this function but this is not important for our purpose.

Then Collatz looked at successive applications of $C$ on any natural number $n$. Such a way, one can build a sequence $(n, C(n), C^2(n), \cdots)$ issued from any not null natural number. It seems that all these sequences contains at least one time the natural number 1. Numerous numerical tests failed to find a counter-example until today.

\subsection{Such a proof by dichotomy of Collatz problem would validate it}
On the one hand, let us suppose that someone establish a proof of Collatz problem by dichotomy driven by parity of euclidean division by $2^k$. Theorem~\ref{thm:general} states that at most one natural number $n$ does not satisfy $P$.

As a first approach of this case, it is easy to verify that $0$ denies Collatz problem. Effectively, $C(0)=0/2 = 0$ as 0 is even. Then Collatz problem would be positively solved for every non null natural numbers under the condition that one could establish such a proof by dichotomy. 

Another way to prove Collatz problem from the supposition of the existence of a proof by dichotomy driven by the euclidean division by $2^k$, consists to suppose that there exist a counter example $n\neq 0$ that doesn't satisfy Collatz problem. 
Theorem~\ref{thm:general} states that it would be the only one. But the sub-sequence starting from $C(n)$ won't neither contain 1. As $C(n)$ cannot be equal to $n$ when $n\neq 0$ by definition of $C$, there should exist at least two natural numbers denying Collatz problem. This denies theorem~\ref{thm:general}, refuting by the same way the existence of such an $n$.

A more precise look at what we talked about would be such a proof shows that as 0 would be the only counter-example. As a consequence we have to prove that Collatz problem is verified for any odd number at the first step. This affirmation is well-known for a long time and didn't bring any solution until today. So, we do not expect the existence of such a proof to solve Collatz problem. Just let us say that these two deductions from the existence of such a proof on this problem are more samples of how to use this method of proof than trying to solve Collatz problem this way.

\section{Conclusion and perspectives}
We defined what we mean by a proof of a proposition $P$ by dichotomy driven by the euclidean division by $2^k$. We stated that at most one natural number may deny $P$ under the condition that such a proof exists for $P$. So, this method of proof may be used to prove the validity of some propositions on $\setN$.

In the case of the existence of such a proof for Collatz problem, this problem will be valid for every non null natural numbers. Our last remark on this example doesn't seem to offer a solution to Collatz problem.

A remark we can do at this point is that this method of proof can be quite easily extended to any infinite countable set. Indeed, the definition of a countable set $S$ claims that there exist at least one one-to-one mapping between $\setN$ and $S$. Let us take one. Say that this mapping is $f$ and that $f$ is defined on $\setN$ with values in $S$ for instance. To prove $P(s)$ for any $s\in S$, it is sufficient to prove $P(f(n))$ for any $n\in\setN$. As it is possible to say that $P(f(n))$ is in fact a proposition on $\setN$, a proof by dichotomy driven by euclidean division by $2^k$ may solve the problem in the hypothesis that such a proof exists.

At the beginning of this paper, we defined a proof by dichotomy, then we studied one particular case. Very probably, there exist some other ways to partition $U_k(P)$ giving the same unique possibility of counter-example. We did not explore deeply this possibility. Thus, we wonder if it is possible to find some other possibilities and to make some generalizations of this type of proofs.

Further, a deeper look at the proof of the theorem shows that we never used the fact that the set of unsolved cases is divided in two equally distributed partitions. Then there may exist some others ways to obtain a proof that does not give more than one counter-example.

\end{document}